%% file: enumeratePolytropes.tex
\documentclass[12pt, reqno]{amsart}
\input{preamble}

\def\1{\mathbf{1}}

\usepackage{cite,verbatim}
\author{Ngoc Mai Tran}
\address{Department of Mathematics, University of Texas at Austin, Austin, TX, 78712 and 
Department of Mathematics, University of Bonn, Bonn, Germany 53115
}

\DeclareMathOperator{\TP}{\mathbb{TP}}
\newcommand{\GFP}{\left.\mathcal{GF}_n\right|_{\mathcal{P}}}
\def\tconv{\text{tconv}}
\def\lin{\mathsf{lin}}
\def\Z{\mathbb{Z}}
\newcommand\ini{\text{in}}

\title{Enumerating polytropes}
\begin{document}

\begin{abstract}
Polytropes are both ordinary and tropical polytopes. We show that tropical types of polytropes in $\mathbb{TP}^{n-1}$ are in bijection with cones of a certain Gr\"{o}bner fan $\mathcal{GF}_n$ in $\R^{n^2 - n}$ restricted to a small cone called the polytrope region. These in turn are indexed by compatible sets of bipartite and triangle binomials. Geometrically, on the polytrope region, $\mathcal{GF}_n$ is the refinement of two fans: the fan of linearity of the polytrope map appeared in \cite{tran.combi}, and the bipartite binomial fan. This gives two algorithms for enumerating tropical types of polytropes: one via a general Gr\"obner fan software such as \textsf{gfan}, and another via checking compatibility of systems of bipartite and triangle binomials. We use these algorithms to compute types of full-dimensional polytropes for $n = 4$, and maximal polytropes for $n = 5$.
\end{abstract}
\maketitle
%, a coarsening of a linear isomorphism of the braid arrangement
%TODO: mention what the other cones of \GF_n are: tropical polytopes whose polytrope is in \TP^{n-1}.

\section{Introduction}

Consider the tropical min-plus algebra $(\R, \oplus, \odot)$, where $a \oplus b = \min(a,b)$, $a\odot b = a+b$. A set $S \subset \R^n$ is tropically convex if $x,y \in S$ implies $a\odot x \oplus b \odot y \in S$ for all $a,b \in \R$. Such sets are closed under tropical scalar multiplication: if $x \in S$, then $a \odot x \in S$. Thus, one identifies tropically convex sets in $\R^n$ with their images in the tropical affine space $\mathbb{TP}^{n-1} = \R^n \backslash (1, \ldots, 1) \R$. The tropical convex hull of finitely many points in $\mathbb{TP}^{n-1}$ is a \emph{tropical polytope}. A tropical polytope is a \emph{polytrope} if it is also an ordinary convex set in~$\mathbb{TP}^{n-1}$ \cite{JoswigK10}.

Polytropes are important in tropical geometry and combinatorics. They have appeared in a variety of context, from hyperplane arrangements \cite{lampostnikov}, affine buildings \cite{affinebuilding}, to tropical eigenspaces, tropical modules \cite{butkovic, BCOQ}, and, semigroup of tropical matrices \cite{kambite}, to name a few. Their discovery and re-discovery in different contexts have granted them many names: they are the alcoved polytopes of type A of Lam and Postnikov \cite{lampostnikov}, the bounded $L$-convex sets of Murota \cite[\S 5]{murota2003discrete}, the image of Kleene stars in tropical linear algebra \cite{butkovic, BCOQ}.  In particular, they are building blocks for tropical polytopes: any tropical polytope can be decomposed into a union of cells, each is a polytrope \cite{mikebernd}. Each cell has a \emph{type}, and together they define the type of tropical polytope. 
A $d$-dimensional polytrope has exactly one $d$-dimensional cell, namely, its (relative) interior. This is the \emph{basic cell}, and its type is the \emph{tropical type} of the polytrope \cite{JoswigK10}. We use the word `tropical' to distinguish from the ordinary combinatorial type defined by the face poset. As we shall show, tropical type refines ordinary type. 

This work enumerates the tropical types of full-dimensional polytropes in $\mathbb{TP}^{n-1}$. Since polytropes are special tropical simplices \cite[Theorem 7]{JoswigK10} this number is at most the number of regular polyhedral subdivisions of $\Delta_{n-1} \times \Delta_{n-1}$ by \cite[Theorem 1]{mikebernd}. However, this is a very loose bound, the actual number of types of polytropes is much smaller. Joswig and Kulas \cite{JoswigK10} pioneered the explicit computation of types of polytropes in $\mathbb{TP}^2$ and $\mathbb{TP}^3$ using the software \textsf{polymake}. They started from the smallest polytrope, which is a particular ordinary simplex \cite{JoswigK10}, and recursively added more vertices in various tropical halfspaces. Their table of results and beautiful figures have been the source of inspiration for this work. Unfortunately, the published table in \cite{JoswigK10} has errors. For example, there are six, not five, distinct tropical types of full-dimensional polytropes in $\mathbb{TP}^3$ with maximal number of vertices, as discovered by Jim{\'e}nez and de la Puente \cite{jimenez2012six}. We recomputed Joswig and Kulas' result in Table \ref{tab:jk.table}.

In contrast to previous works \cite{JoswigK10, jimenez2012six}, we have a Gr\"{o}bner approach polytropes. In Section \ref{sec:background}, we show that their tropical types are in bijection with a subset of cones in the Gr\"{o}bner fan $\mathcal{GF}_n$ of a certain toric ideal. While this is folklore to experts, the obstacle has been in characterizing these cones. Without such characterizations, brute force enumeration requires one to compute all of $\mathcal{GF}_n$. Unfortunately, even with symmetry taken into account, $\mathcal{GF}_5$ cannot be handled by leading software such as \textsf{gfan} \cite{gfan} on a conventional desktop. 

We show that the full-dimensional polytrope cones in $\mathcal{GF}_n$ are contained in a small cone called the \emph{polytrope region}. Our main result, Theorem \ref{thm:S}, gives an indexing system for the polytrope cones in terms of sets of compatible bipartite binomials and triangles. Geometrically, we show that on the polytrope region, the fan $\mathcal{GF}_n$ equals the refinement of
the fan of linearity of the polytrope map $\mathcal{P}_n$, and the bipartite binomial fan $\mathcal{BB}_n$. The later fan is constructed as a refinement of finitely many fans, each is the coarsening of an arrangement linearly isomorphic to the braid arrangement. The open, full-dimensional cones are in bijection with polytropes in $\mathbb{TP}^{n-1}$ with maximal number of vertices. 
These results elucidate the structure of $\mathcal{GF}_n$ and gives algorithms for polytrope enumeration. Specifically, one can either compute the Gr\"{o}bner fan $\mathcal{GF}_n$ restricted to the polytrope region, or enumerate sets of compatible bipartite binomials and triangles. With these approaches, we computed representatives for all tropical types of full-dimensional polytropes in $\mathbb{TP}^3$ and all maximal polytropes in $\mathbb{TP}^4$. In $\mathbb{TP}^4$, up to permutation by $\mathbb{S}_5$, there are~$27248$ tropical types of maximal polytropes. This is the first result on tropical types of polytropes in dimension~4. \footnote{An earlier version of reported maximal polytropes in $\mathbb{TP}^5$. Unfortunately, in fact, the computation ran out of memory and reported an erroneous number. We thank Michael Joswig and his team for pointing this out. The number of tropical types of polytropes in dimension 5 is still open.}

\textbf{Organizations.} For self-containment, Section \ref{sec:background} reviews the basics of Gr\"{o}bner bases and integer programming, and the three integer programs central to this paper. Section \ref{sec:tropical.polytope} revisits the result of Develin and Sturmfels \cite{mikebernd} on types of tropical polytopes using Gr\"obner bases. We use this view in Section \ref{sec:polytropes} to derive Theorem \ref{thm:ds.analogue}, the analogue of Develin and Sturmfels main results for polytropes. Section \ref{sec:main} contains our main result, Theorem \ref{thm:S} and \ref{thm:main}, on the structure of the polytrope complex. Section~\ref{sec:algorithm} presents algorithms for enumerating full-dimensional polytropes, as well as computation results for $\mathbb{TP}^3$ and $\mathbb{TP}^4$. We conclude with discussions and open problems.

\textbf{Notation.} Throughout this text, for a positive integer $n$, let $[n]$ denote the set $\{1, \ldots, n\}$. We shall identify an $n \times m$ matrix $c$ with the vector $(c_{ij}, i \in [m], j \in [n])$ of length $nm$.
If $c$ is an $n \times n$ matrix with zero diagonal, identify it with the vector $(c_{ij}, i \in [n], j \in [n], j \neq i)$ of length $n^2 - n$. For a cone $\mathcal{C}$, let $\mathcal{C}^\circ$ denote its relative interior, $\partial \mathcal{C}$ denote its boundary. 
%Identify a monomial $x^u \in \R[x_{ij}: ij \in [N]]$ with its exponent $u \in \mathbb{N}^N$, as well as the multigraph on $n$ nodes where $u_{ij}$ gives the number of edges from $i$ to $j$. The \emph{support of $u$} is the set of indices $ij$ where $u_{ij} > 0$. Identify a binomial $x^{u_+} - x^{u_-}$ with $u = u_+ - u_- \in \mathbb{Z}^N$, where  $u_+ = \max(u, 0) \in~\mathbb{N}^n$, $u_- = -\min(u, 0) \in \mathbb{N}^n$. Where convenient, identify $u \in \R^{n^2-n}$ with the matrix $u \in \R^{n \times n}$ with zero diagonal, viewed as the weight matrix of a graph on $n$ nodes. 
%The support of a vector $u \in \mathbb{N}^n$ is the set of indices $i \in [n]$ where $u_i > 0$. The support of a polyhedral complex is the set of points that belongs to that complex. 
% Say that $u$ is \emph{primitive} if the greatest common divisor of the $u_{ij}$'s is 1. 

\section{Background}\label{sec:background}

This section contains a short exposition on the Gr\"{o}bner approach to integer programming, adapted from \cite[\S 5]{st94}. Another excellent treatment from the viewpoint of applied algebraic geometry is \cite[\S 8]{usingAlgebraicGeometry}, while \cite[\S 9]{triangulationBook} approaches the topic from triangulations of point configurations. 

\subsection{Gr\"obner fan and integer programs}
For $c \in \R^m$, $A \in \R^{n \times m}$ and $b \in \mathbb{Z}^n$, the primal and dual of an integer program are
\begin{align}
\mbox{minimize } & c^\top u  \label{lp.p} \tag{P} \\
\mbox{ subject to } & Au = b, \hspace{1em} u \in \mathbb{N}^N \notag \\
    \text{maximize }   & b^\top y  \label{lp.d} \tag{D} \\
    \text{subject to } & A^\top y \leq c, y \in \mathbb{R}^n. \notag
\end{align}
%TODO: write a blurb saying where we are going with this. We have a parametrized family of integer programs, one for each cost and constraint pair $(c,b)$. Grobner bases, via Conti-Traverso, is one way to solve integer programs. Has been used to study parametrized programs. More importantly, get a combinatorial equivalence which refines combinatorial type of the constraint polytope in \ref{lp.d}. For the case of tropical polytopes, associated with the transport program, this notion coincides with a natural notion of type for tropical polytopes. 
Consider the polynomial ring $\mathbb{R}[x] = \mathbb{R}[x_1, \ldots, x_m]$. Identify $u \in \mathbb{N}^m$ with the monomial $x^u = \ds\prod_{i \in [m]} x_i^{u_i}$ in $\mathbb{R}[x]$. The \emph{toric ideal of $A$} is 
$$ I = \langle x^u - x^v: Au = Av, u,v \in \mathbb{N}^m\rangle. $$
%That is, this is the ideal generated by binomials which correspond to pairs of competing solutions to the primal program (\ref{lp.p}). 
%Let $\succ$ denote the lexicographic order on the monomials in $\R[x]$. 
%\subsection{Term ordering and Gr\"obner fan}
Let $c \in \R^m$ be a cost vector. The \emph{term ordering} $\succ_c$ is a partial order on the monomials in $\mathbb{R}[x]$, defined as
$$ x^u \succ_c x^v, u \succ_c v \hspace{0.5em} \Leftrightarrow \hspace{0.5em}  c \cdot u > c \cdot v. $$
% \mbox{ or } c \cdot u = c \cdot v \mbox{ and $x^u \succ x^v$}.  $$
For a polynomial $f = \sum_u a_u x^u \in \R[x]$, define its initial form $\ini_c(f)$ to be the sum of all terms $a_ux^u$ with maximal order under $\succ_c$. The \emph{initial ideal} of $I$ is the ideal generated by $\ini_c(f)$ for all $f \in I$:
$$ in_c(I) = \langle in_c(f): f \in I \rangle. $$
Monomials of $I$ which do not lie in $in_c(I)$ are called \emph{standard monomials}. 

Now we consider $c \in \R^m$ up to their initial ideals $\ini_c(I)$. Let $\mathcal{C}_c(I) \subseteq \R^m$ be the equivalence class containing $c$
$$ \mathcal{C}_c(I) := \{c' \in \R^m: \ini_{c'}(I) = \ini_c(I)\}. $$
 In general, $\mathcal{C}_c(I)$ may not be a nice set - for example, it may not be convex \cite{fukuda2007computing}. When $c \in \R^m_{> 0}$, $\mathcal{C}_c(I)$ is convex, and its closure is a polyhedral cone \cite{fukuda2007computing}. Following \cite{fukuda2007computing}, define the \emph{Gr\"obner fan of $I$}, to be the collection of closed cones $\overline{\mathcal{C}_c(I)}$ where $c \in \R^{m}_{> 0}$ together with all their non-empty faces. The support of the Gr\"obner fan is called the  Gr\"{o}bner region
$$ \bigcup_{c \in \R^{m}_{> 0}} \overline{\mathcal{C}_c(I)}. $$
If $I$ is homogeneous, then the Gr\"obner region equals $\R^m$ \cite{st94}. If $I$ is not homogeneous, one can homogenize it. Each homogenized version of $I$ is the toric ideal of some matrix $A^h$, called the lift of $A$. This matrix has the form
$$ A^h = \left[ \begin{array}{cc} A & \mathbf{0} \\ \mathbf{1} & \mathbf{1} \end{array} \right], $$
where $\textbf{0}$ is a zero matrix, and $\mathbf{1}$'s are matrices of all ones of appropriate sizes. 

%\subsection{Gr\"obner bases}
A \emph{Gr\"{o}bner basis of $I$} with term ordering $\succ_c$ is a finite subset $S_c \subset I$ such that $\{\ini_c(g): g \in S_c\}$ generates $\ini_c(I)$. It is called minimal if no polynomial $\ini_c(g)$ is a redundant generator of $\ini_c(I)$. It is called reduced if for any two distinct elements $g, g' \in S_c$, no monomial of $g'$ is divisible by $\ini_c(g)$. A \emph{universal Gr\"obner basis of $I$} is a set $S$ that is a Gr\"obner basis with respect to any term ordering $\succ_c$. 

Throughout this paper we will only be concerned with three integer programs whose matrices $A$ are totally unimodular, that is, every minor of $A$ is either $+1, 0$ or $-1$. Such a matrix has a number of nice properties. In particular, define a circuit of $A$ to be a non-zero primitive vector $u$ in the kernel of $A$ with minimal support with respect to set inclusion. If $A$ is totally unimodular, the set
$$ \{x^{u_+} - x^{u_-}: u \mbox{ is a circuit of} A \} $$
is a universal Gr\"obner basis of $I$ \cite[Theorem 5.9]{rekha}. Furthermore, the Gr\"obner fan of $A$ coincides with the secondary fan of $A$, which is dual to regular subdivisions of the configuration of points that are the columns of $A$. 

\subsection{The transport program}
Throughout this paper, we shall be concerned with the transport program and two of its variants, the all-pairs shortest path, and the homogenized all-pairs shortest path programs. These classic integer programs play central roles in defining and understanding tropical types of polytropes, as we shall discuss in the following sections. 

Fix $c \in \R^{n \times m}$ and $b \in \Z^{n + m}$. With variables $u \in \mathbb{N}^{n \times m}$, $y \in \R^n$, $z \in \R^m$, the transport program is
\begin{align}
\text{minimize }   & \sum_{i\in[n],j\in[m]} u_{ij}c_{ij}  \label{transport.p} \tag{P-transport} \\
\text{subject to } & \sum_{j\in [m]}u_{ij} = b_i, \sum_{i\in[n]} u_{ij} = b_j \hspace{1em} \mbox{ for all } i \in [n], j \in [m]. \notag \\
\text{maximize }   & \sum_{i\in[n]}y_ib_i - \sum_{j\in[m]}z_jb_j  \label{transport.d} \tag{D-transport} \\
\text{subject to } & y_i - z_j \leq c_{ij}, \hspace{1em} \mbox{ for all } i \in [n], j \in [m] \notag
\end{align}

This program defines a transport problem on a directed bipartite graph on $(m,n)$ vertices. Here $c_{ij}$ is the cost to transport an item from $i$ to $j$, $b_i$ is the number of items that node $i$ want to sell, $b_j$ is the number of item that node $j$ want to buy, $u_{ij}$ is the number of items to be sent from $i$ to $j$, and $y_i$, $z_j$ are the per-item prices at each node. The primal goal is to choose a transport plan $u \in \Z^{m + n}$ that minimizes costs and meets the targeted sales $b$. The dual goal is to set prices to maximize profit, subject to the transport cost constraint. 

The toric ideal associated to this program is
\begin{equation}\label{eqn:it}
I_t = \langle x^u - x^v: \sum_j u_{ij} = \sum_j v_{ij}, \sum_iu_{ij} = \sum_iv_{ij} \mbox{ for all } i \in [m], j \in [n]
\rangle.
\end{equation}
Here the subscript $t$ stands for `transport'. This ideal plays a central role in classification of tropical polytopes, as we shall discuss in Section \ref{sec:tropical.polytope}. 

\subsection{The all-pairs shortest path program}

This is the transport program with $m = n$ and $z = -y$, and cost matrix $c \in \R^{n \times n}$ with $c_{ii} = 0$ for all $i \in [n]$. Explicitly, fix such a cost matrix $c$ and constraint vector $b \in \mathbb{Z}^n$. With variables $u \in \mathbb{N}^{n \times n}$ where $u_{ii} = 0$ for all $i \in [n]$, and $y \in \R^n$, the all-pairs shortest path program is 

\begin{align}
\text{minimize }   & \sum_{i,j\in[n]} u_{ij}c_{ij}  \label{shortest.p} \tag{P-shortest} \\
\text{subject to } &  \sum_{j=1}^n u_{ij} - \sum_{j=1^n} u_{ji} = b_i \mbox{ for all } i = 1, \ldots, n. \label{eqn:A} \\
\text{maximize }   & \sum_{i=1}^nb_iy_i \label{shortest.d} \tag{D-shortest} \\
\text{subject to } & y_i - y_j \leq c_{ij}, \hspace{1em} \mbox{ for all } i,j\in [n], i \neq j. \label{eqn:lp.pol}
\end{align}

Here one has a simple directed graph on $n$ nodes with no self loops. As before, $b$ is the targeted sales, $c$ is the cost matrix, $u$ defines a transport plan, $y$ is the price vector. 
Note that in this problem, each node can both receive and send out items. 

The all-pairs shortest path is a basic problem in integer programming. It appears in a variety of applications, one of which is classification of polytropes (cf. Section \ref{sec:polytropes}). We collect some necessary facts about this program below. These properties can be found in \cite[\S 4]{networkBook}. See \cite[\S 3]{BCOQ} and \cite[\S4]{butkovic} for treatments in terms of tropical eigenspaces. 

\subsubsection{Feasible region, lineality space}
This program is feasible only if $\sum_ib_i = 0$ and $c$ has no negative cycles. Let $R_n$ denote the set of feasible cost matrices $c$. Then
\begin{equation} \label{eqn:recessionCone}
R_n = \{c \in \R^{n^2-n}:  c \cdot \chi_\omega \geq 0\}
\end{equation}
where $\chi_\omega$ is the incidence vector of the cycle $\omega$ and $\omega$ ranges over all simple cycles on $n$ nodes. Explicitly, for a cycle $\omega = i_1 \to i_2 \to \ldots \to i_k \to i_1$, 
$$ c_{i_1i_2} + c_{i_2i_3} + \ldots + c_{i_ki_1} \geq 0. $$
The feasible region $R_n$ is a closed cone in $\R^{n^2-n}$. Note that if $c \in R_n$, then $c + c' \in R_n$ for any matrix $c'$ such that $c' \cdot \chi_\omega = 0$ for all cycles $\omega$. One says that the set of such $c'$ forms the lineality space of $R_n$, $\lin(R_n)$
\begin{equation}\label{eqn:vn}
\lin(R_n) = \{c \in \R^{n^2-n}: c \cdot \chi_\omega = 0\}.
\end{equation}
This is an $(n-1)$ dimensional space, consisting of matrices of the form $c_{ij} = s_i - s_j$ for some $s \in \R^n$. This is the space of flows, with gradient vectors $s$. It is also known as the space of strongly consistent matrices in pairwise ranking theory, with $s_i$ interpreted as the score of item $i$ \cite{saari,lekheng}.

\subsubsection{Kleene stars}
To send an item from $i$ to $j$, one can use the path $i \to j$ with cost $c_{ij}$, or the path $i \to k \to j$ with cost $c_{ik} + c_{kj}$, and so on. This shows up in the constraint set (\ref{eqn:lp.pol}): for any triple $i,j,k$, we have $ y_i - y_j = (y_i - y_k) + (y_k - y_j)$, so in addition to $y_i - y_j \leq c_{ij}$, we also have $ y_i - y_j \leq c_{ik} + c_{kj}$, and by induction, $y_i - y_j$ is less than the cost of any path from $i$ to $j$. Thus, the constraint $y_i - y_j \leq c_{ij}$ is tight if and only if $c_{ij}$ is the shortest path from $i$ to $j$. If we assume $c$ has no negative cycle, then the shortest path has finite value. This motivates the following definition. 
%Note that if $c$ has a negative cycle (ie: a cycle whose cost is negative), then a path from $i$ to $j$ can visit this cycle infinitely many times and attain a large negative value. In this case, the shortest path from $i$ to $j$ has value $-\infty$, and the all-pairs shortest path program is infeasible. 

\begin{defn}\label{defn:kleene.rn}
For $c \in R_n$, the \emph{Kleene star of $c$} is $c^\ast \in \R^{n^2-n}$ where $c^\ast_{ij}$ is the weight of the shortest path from $i$ to $j$.
\end{defn}

%\subsubsection{The Kleene star map and fan of linearity}
To avoid saying `the constraint set of an all-pairs shortest path dual program with given $c$' all the time, we shall call this set the \emph{polytrope of $c$}. Justification for this terminology comes from Proposition \ref{prop:polytrope.generators} in Section \ref{sec:polytropes}. 

\begin{defn}[Polytrope of a matrix]\label{defn:pol.c}
Let $c \in R_n$. The polytrope of $c$, denoted $Pol(c)$, is the set
\begin{equation}\label{eqn:pol.c}
Pol(c) = \{y \in \R^n: y_i - y_j \leq c_{ij} \forall i,j \in [n], i \neq j\}.
\end{equation}
\end{defn}
As discussed above, one can always replace $c$ by $c^\ast$ in the facet description of the polytrope of $c$ and not change the set. 
\begin{cor}
For $c \in R_n$, $Pol(c) = Pol(c^\ast)$. 
\end{cor}

\begin{defn}
The \emph{polytrope region} is
$$ \mathcal{P}_n = \{c \in R_n: c = c^\ast\} \subset \R^{n^2 - n}. $$
\end{defn}
The polytrope region $\mathcal{P}_n$ is a closed cone in $\R^{n^2 - n}$. It is also known as the set of distance matrices $c \in \R^{n \times n}$, since it can be identified with the set of matrices with zero diagonal that satisfy the triangle inequality
$$ \mathcal{P}_n \cong \{c \in \R^{n \times n}: c_{ii} = 0, c_{ij} \leq c_{ik} + c_{kj} \mbox{ for all } i,j,k \in [n]\}. $$
The map $c \mapsto (c^\ast_{ij}, i,j \in [n])$ is piecewise linear in each entry. Domains where this map is given by a linear functional for each $i,j \in [n]$ form cones of $\R^{n^2-n}$, and altogether they form the fan of linearity of the polytrope map studied in \cite{tran.combi}. Restricted to the polytrope region, this fan is a polyhedral complex, which we shall also denote $\mathcal{P}_n$. 

\subsubsection{Toric ideal}
Let $I_s$ be the toric ideal associated with the all-pairs shortest path program. The subscript $s$ standars for `shortest path'. As before, we suppress the dependence on $n$ in the notation. This ideal can be written explicitly as
$$
I_s = \langle x_{ij}x_{ji} - 1, x_{ij}x_{jk} - x_{ik} \rangle
$$
where the indices range over all distinct $i,j,k \in [n]$. Write the primal all-pairs shortest path program in standard form, and let $A_s$ be the corresponding matrix that defines the constraint set of the primal. Then $A_s$ is totally unimodular \cite{networkBook}. In particular, $I_s$ is generated by binomials $x^{u_+} - x^{u_-}$, where $(u_+,u_-)$ is a circuit of $A_s$. As we shall see in Section \ref{sec:main}, a subset of these circuits are crucial for enumeration of polytropes up to their tropical types.

The Gr\"obner fan of $I_s$ is the central object of study in our paper. We shall write $\mathcal{GF}_n$ for the Gr\"obner fan of $I_s$, emphasizing the dimension. We collect some facts about $\mathcal{GF}_n$

\begin{lem}\label{lem:linspace.gf}
The lineality space of $\mathcal{GF}_n$ is $\lin(R_n)$ defined in (\ref{eqn:vn}).
\end{lem}
\begin{proof}
Let $\mathcal{C}$ be a cone of $\mathcal{GF}_n$. Take $c \in \mathcal{C}$. For $[s_i - s_j] \in \lin(R_n)$, consider $\bar{c} = c + [s_i - s_j]$. That is,
$$ \bar{c}_{ij} = c_{ij} - s_i + s_j. $$
Now, for any cycle $\omega$, $c \cdot \chi_\omega = \bar{c} \cdot \chi_\omega$. Thus for any circuit $(u_+, u_-)$, $c \cdot (u_+ - u_-) = \bar{c} \cdot (u_+ - u_-)$, so $c \cdot u_+ \geq c \cdot u_-$ if and only if $\bar{c} \cdot u_+ \geq \bar{c} \cdot u_-$. Since the program (\ref{lp.p}) is unimodular, the ideal $I_s$ is generated by circuits. Thus, the term orders $\succ_c$ and $\succ_{\bar{c}}$ are equal, so $\bar{c} \in \mathcal{C}$. That is, every cone $\mathcal{C}$ of $\mathcal{GF}_n$ has lineality space $\lin(R_n)$, so $\mathcal{GF}_n$ has lineality space $\lin(R_n)$.
\end{proof}

\begin{lem}\label{lem:grobner.region}
The Gr\"{o}bner region of $\mathcal{GF}_n$ is $R_n$ defined in (\ref{eqn:recessionCone}).
\end{lem}
\begin{proof}
As mentioned, $R_n$ is the feasible region of the integer program (\ref{lp.p}), and thus contains the Gr\"obner region. To show the reverse inclusion, take $c \in R_n$. We need to show that the Gr\"obner cone of $c$ contains a point in the the positive orthant $\mathbb{R}_{\geq 0}^{n^2-n}$. Indeed, let $y \in Pol(c)$. Define $\bar{c}$ via
$$\bar{c}_{ij} = c_{ij} - y_i + y_j.$$
Since $y \in Pol(c)$, $c_{ij} \geq y_i - y_j$, so $\bar{c} \in \mathbb{R}_{\geq 0}^{n^2-n}$. By Lemma \ref{lem:linspace.gf}, $\bar{c}$ belongs to the same cone in $\mathcal{GF}_n(I_s)$ as~$c$. So $\bar{c}$ is the point needed.
\end{proof}

%In particular, $\mathcal{GF}_n$ is a fan in $\R^{n^2-n}$ with lineality space of dimension $n-1$.

\subsection{The homogenized all-pairs shortest path}
Identify $c \in \R^{n^2-n}$ with its matrix form in $\R^{n \times n}$, where $c_{ii} = 0$ for all $i \in [n]$. So far, we have only defined Kleene stars for $c \in \R^{n \times n}$ with zero-diagonal and non-negative cycles. We now extend the definition of Kleene stars to general matrices $c \in \R^{n \times n}$. This leads to the problem of weighted shortest paths. In the tropical linear algebra literature, one often goes the other way around: first consider the weighted shortest path problem, derive Kleene stars for general matrices $c$, and then restricts to those in $R_n$ (see \cite{SS08, tran.combi, butkovic, AGG09, kambite, BCOQ}). The reverse formulation, from feasible shortest paths to weighted shortest paths, is not so immediate. However, in the language of Gr\"{o}bner bases, this is a very simple and natural operation: making the fan $\mathcal{GF}_n$ complete by homogenizing $I_s$.

Introduce $n$ variables $x_{11},x_{22}, \ldots, x_{nn}$. Consider the following homogenized version of $I_s$ in the ring $\mathbb{R}[x_{ij}:i,j = 1, \ldots, n]$
$$ I_s^h = \langle x_{ij}x_{ji} - x_{ii}x_{jj}, x_{ij}x_{jk} - x_{ik}x_{kk}, x_{ii} - x_{jj} \rangle $$
where the indices range over all distinct $i,j,k \in [n]$. This is the toric ideal of the following program
\begin{align}
\mbox{minimize } & \sum_{i,j \in [n]} c_{ij}u_{ij}  \label{lp.p.prime} \tag{$\mathrm{P^h-shortest}$} \\
\mbox{ subject to } & \sum_{j\neq i, j=1}^n u_{ij} - \sum_{j\neq i, j = 1}^n u_{ji} = b_i \mbox{ for all } i = 1, \ldots, n. \notag \\
& \sum_{i=1}^n\sum_{j=1}^n u_{ij} = b_{n+1}. \notag
\end{align}
Compared to (\ref{shortest.d}), the dual program of (\ref{lp.p.prime}) has one extra variable. It is helpful to keep track of this variable separately. Let $\lambda \in \R$. Write $b^\top = (b_1 \, \ldots \, b_n)$.  The dual program to (\ref{lp.p.prime}) is the following. 
\begin{align}
    \text{maximize }   & b^\top y + b_{n+1}\lambda  \label{lp.d.prime} \tag{$\mathrm{D^h-shortest}$} \\
    \text{subject to } & y_i - y_j - \lambda \leq c_{ij} \mbox{ for all } i,j \in [n] \notag \\
& \lambda \geq c_{ii} \mbox{ for all } i \in [n]. \notag
\end{align}

In fact, $\lambda$ and $y$ can be solved separately. For example, by adding the constraints involving $c_{ij}$ and $c_{ji}$, we obtain a constraint only in $\lambda$
$$ (y_i - y_j) - \lambda + (y_j - y_i) - \lambda \leq c_{ij} + c_{ji}, \Leftrightarrow \lambda \geq \frac{c_{ij} + c_{ji}}{2}. $$
More systematically, set $b$ to be the all-zero vector, $b_{n+1} = 1$, and view the primal program as a linear program over $\mathbb{Q}$. (We can always do this, as there are finitely many decision variables). Then the dual program (\ref{lp.d.prime}) has optimal value $\lambda$. The corresponding primal program becomes
\begin{equation}\label{lp.lambda}
\begin{matrix}
& {\rm Minimize} \,\, \sum_{i,j=1}^n c_{ij} u_{ij}
\,\, \,\, \hbox{subject to} \,\,\,\, u_{ij} \geq 0
\,\,\hbox{ for } 1 \leq i,j \leq n , \\ &
 \sum_{i,j=1}^n  u_{ij} = 1  \,\,\,\, \hbox{and}\,\,\,
\sum_{j=1}^n u_{ij} = \sum_{k=1}^n u_{ki}
\,\,\hbox{ for all }\, 1 \leq i \leq n.
\end{matrix}
\end{equation}

This program first appeared in \cite{Cg62}. The constraints require $(u_{ij})$ to be a probability distribution on the edges of the graph of $c$ that represents a flow. The set of feasible solutions is a convex polytope called the \emph{normalized cycle polytope}. Its vertices are the uniform probability distributions on directed cycles. %This polytope lives in the affine subspace that is the direct sum of $V_n$ defined in (\ref{eqn:vn}), and the affine subspace 
%$$ \mathcal{H}_{-1} = \{u \in \R^{n^2-n}: \sum_{i,j=1}^n u_{ij} = 1\}. $$
By strong duality, $\lambda$ is precisely the value of the minimum normalized cycle in the graph weighted by $c$. Plugging in this value for $\lambda$, we find that (\ref{lp.p.prime}) is the original all-pairs shortest path problem with new constraints $c'_{ij} = c_{ij} - \lambda$, $c'_{ii} = 0$ for all $i,j \in [n]$. This tells us how to define the Kleene star of $c$. 

\begin{defn}\label{defn:kleene.general}
Let $c \in \R^{n \times n}$. Let $\lambda(c)$ be the value of the minimum normalized cycle in the graph weighted by $c$. Define $c' \in \R^{n \times n}$ via $c'_{ij} = c_{ij} - \lambda(c)$, $c'_{ii} = 0$. The Kleene star of $c$, denoted $c^\ast$, is the $n \times n$ matrix such that $c^\ast_{ij}$ is the shortest path from $i$ to $j$ in the graph with edge weights $c'$.
\end{defn}

This definition reduces to the Kleene star in Definition \ref{defn:kleene.rn} when $c \in R_n$, so in this sense it is an extension of Definition \ref{defn:kleene.rn} to general $n \times n$ matrices. The value $\lambda(c)$ is the tropical eigenvalue of the matrix $c$, and the polytope defined as the constraint set of (\ref{lp.d}) with $(c')^\ast$ is called the tropical eigenspace of $c$. As the names suggested, these objects play important roles in the spectral theory of tropical matrices, see the monographs \cite{butkovic, BCOQ} for key results in this field.

%As mentioned in the introduction, this is how Kleene stars are defined in the literature. 

%Let $\mathcal{N}_n$ denote the normal fan of the normalized cycle polytope. The cone $R_n$ can be identified with the cone of $\mathcal{N}_n$ of codimension $n-1$, indexed by $n$ self-loops, one at each node $i \in [n]$. 

\section{Tropical polytopes and their types}\label{sec:tropical.polytope}

In this section we define tropical polytopes, and review the main theorem of \cite{mikebernd} in terms of the transport problem. Say that a set $P \subset \R^n$ is closed under scalar tropical multiplication if $x \in P$ implies $\lambda \odot x = (\lambda + x_1, \ldots, \lambda + x_n) \in P$ for all $\lambda \in \R$. Such a set can also be regarded as a subset of $\mathbb{TP}^{n-1}$. We will often identify $\mathbb{TP}^{n-1}$ with $\R^{n-1}$. Say that $P \subset \mathbb{TP}^{n-1}$ is a classical polytope if it is a polytope in $\R^{n-1}$ under this identification.
 %via the invertible map
%$$ \natural: \mathbb{TP}^{n-1} \to \R^{n-1}, (x_1, \ldots, x_n) \mapsto (x_2-x_1, \ldots, x_n - x_1). $$
%In general, we use the superscript $\natural$ (read `natural') to denote the image of sets under this map. Say that $P \subset \mathbb{TP}^{n-1}$ is a classical polytope if $C^\natural$ is a polytope in $\R^{n-1}$. 
A tropical polytope in $\R^n$ is the tropical convex hull of $m$ points $c_1, \ldots, c_m \in \R^n$
\begin{align*}
\text{tconv}(c_1, \ldots, c_m) &= \{z_1\odot c_1 \oplus \ldots \oplus z_m\odot c_m: z_1,\ldots,z_n \in \R\} \\
&= \{\min(z_1+c_1, \ldots, z_m+c_m): z_1,\ldots,z_m \in \R\}.
\end{align*}
For $c$ an $n \times m$ matrix with columns $c_1, \ldots, c_m$, we will write $\text{tconv}(c)$ for
$\tconv(c_1, \ldots, c_m)$. Rewritten in the tropical algebra, $\text{tconv}(c)$ is the image set of the matrix $c$.
\begin{equation}\label{eqn:tconv.c}
\text{tconv}(c) = \text{tconv}(c_1, \ldots, c_m) = \{y \in \R^n: y = c\odot z \mbox{ for some } z \in \R^m \}.
\end{equation}
 Note that a tropical polytope $\tconv(c)$ is closed under scalar tropical multiplication, that is, $\tconv(c) \subseteq \mathbb{TP}^{n-1}$. 

Develin and Sturmfels \cite{mikebernd} pioneered the investigation on tropical polytopes. They showed \cite[Lemma 22]{mikebernd} that $\tconv(c)$ is a union of bounded cells. In particular, let $Q_c$ be the constraint set of the dual transport program (\ref{transport.d})
$$ Q_c = \{(y,z) : y_i - z_j \leq c_{ij}, i \in [n], j \in [m] \}. $$
Then each cell of $\tconv(c)$ is the projection onto the $y$ coordinate of a bounded face of~$Q_c$. Such a cell has the form
$$ \{y \in \R^n: y_i = c_{ij} + z_j \mbox{ if and only if } S_{ij} = 1, i \in [n], j \in [m]\}$$
for some matrix $S \in \{0,1\}^{n \times m}$, called its \emph{type}. 
\begin{defn}
The type of a tropical polytope $\tconv(c)$ is the set of types of its cells.
\end{defn}

The most effective way to understand cell types of tropical polytopes is via the transport program.
\begin{prop}[\cite{mikebernd}, Lemma 22]\label{prop:mikebernd.grobner}
The tropical polytopes $\tconv(c)$ and $\tconv(c')$ have the same tropical types if and only if $\ini_c(I_t) = \ini_{c'}(I_t)$, where $I_t$ is the transport ideal defined in (\ref{eqn:it}).
\end{prop}

It is worth sketching the idea. The key is to realize that if a bounded face of $Q_c$ is supported by some vector $b$, then the type of the corresponding cell determines the set of optimal transport plans for (\ref{transport.p}) with cost $c$ and constraint $b$, and vice-versa. So $\tconv(c)$ and $\tconv(c')$ have the same tropical type if and only if for each constraint $b \in \Z^{m+n}$, the programs (\ref{transport.p}) with cost $c$ and constraint $b$, and (\ref{transport.p}) with cost $c'$ and constraint $b$ have the same set of optimal transport plans. Now we look at the ideal. Each binomial generator $x^u - x^v$ of $I_t$ is a pair of competing transport plans $(u,v)$ subjected to the same constraint $b_i = \sum_ju_{ij} = \sum_jv_{ij}, b_j = \sum_iu_{ij} = \sum_iv_{ij}$, for some $b \in \mathbb{Z}^n$. Therefore, each polynomial in $I$ consists of at least two monomials, corresponding to competing transport plans. The partial order $\succ_c$ compares plans: if $u \succ_c v$, then $u$ is a strictly worse plan than $v$. Under the transport cost $c$, $\ini_c(I_t)$ is the `ideal of bad plans': if the monomial $x^u \in \ini_c(I_t)$, then $u$ cannot be the optimal plan. Note, however, that $\succ_c$ is only a partial order. So if there are two optimal plans $u,v$ for some constraint $b$, then $x^u - x^v \in \ini_c(I_t)$. The converse is also true: if $x^u - x^v \in \ini_c(I_t)$ but $x^u, x^v \notin \ini_c(I_t)$, then $u$ and $v$ must be two optimal plans. Thus, if $\ini_c(I_t) = \ini_{c'}(I_t)$, then all bad transport plans under the cost matrix $c$ are exactly the same as those under $c'$, and hence all the optimal plans under $c$ and $c'$ agree. So $\ini_c(I_t) = \ini_{c'}(I_t)$ if and only if for each constraint $b \in \Z^{m+n}$, the programs (\ref{transport.p}) with cost $c$ and constraint $b$, and (\ref{transport.p}) with cost $c'$ and constraint $b$ have the same set of optimal transport plans. This is conclusion needed.

The linear program (\ref{transport.p}) is totally unimodular, so the Gr\"obner fan equals the secondary fan of $I_t$. The secondary fan is in bijection with regular subdivision of the point configuration that defines the constraint set of (\ref{transport.p}). In the case of the transport program, this is a product of simplices. So Proposition \ref{prop:mikebernd.grobner} implies the following main theorem of \cite{mikebernd}. 

\begin{thm}[\cite{mikebernd}]\label{thm:ds}
Tropical types of tropical polytopes generated by $m$ points in $\R^n$ are in bijection with regular subdivisions of the product of two simplices $\Delta_{m-1} \times \Delta_{n-1}$.
\end{thm}

%TODO: draw pictures: 5 types of polytropes. 

%They showed that a tropical polytope is a union of bounded cells, each has a \emph{cell type}. Together they specify the \emph{type} of the tropical polytope, see \cite{mikebernd}. The types of tropical simplices are in bijection with cones of a certain Gr\"{o}bner fan \cite[Theorem 1]{mikebernd}, which was further studied in \cite{block2006tropical}. As a corollary, the types of standard tropical simplices are in bijection with cones of $\mathcal{GF}_n$. This furnishes a geometric interpretation of the equivalence relation on $R_n$ induced by initial ideals.

\section{Polytropes and their types}\label{sec:polytropes}

\begin{defn}
A set $P \subset \mathbb{TP}^{n-1}$ is a \emph{polytrope} if $P$ is a tropical polytope and also an ordinary polytope in $\mathbb{TP}^{n-1}$. 
\end{defn}

\begin{defn}
The \emph{dimension} of a polytrope $P$ is the dimension of the smallest affine subspace containing it. Say that $P \subset \mathbb{TP}^{n-1}$ is \emph{full-dimensional} if its dimension is $n-1$. 
\end{defn}

Polytropes have appeared in a variety of contexts. The following classical result states that a polytrope is the constraint set of an all-pairs shorest path dual program~(\ref{lp.d}).
It allows one to write a polytrope $P$ as $Pol(c)$, the polytrope of some unique matrix $c \in \mathcal{P}_n$ $P = Pol(c)$. This justifies why we call $Pol(c)$ the polytrope of $c$ in Definition \ref{defn:pol.c}.

\begin{prop}[\cite{mikebernd, butkovic}]\label{prop:polytrope.generators}
Let $P \subset \mathbb{TP}^{n-1}$ be a non-empty set. The following are equivalent.
\begin{itemize}
  \item $P$ is a polytrope.
  \item There is a unique $c \in \mathcal{P}_n$ such that $P = Pol(c)$, as defined in (\ref{eqn:pol.c}).
  \item There is a unique $c \in \mathcal{P}_n$ such that $P = \tconv(c)$, as defined in (\ref{eqn:tconv.c}).
\end{itemize}
Furthermore, the $c$ in the last two statements are equal.
\end{prop}

Note that we have defined a polytrope $P$ as a set. This creates ambiguity when one speaks of the type of $P$ as a tropical polytope, since the type depends on the choice of generators \emph{and} their orderings. By \cite[Proposition 21]{mikebernd}, every tropical polytope has a unique minimal generating set. A classical result in tropical linear algebra \cite{BCOQ, butkovic} states that a polytrope $P = \tconv(c)$ of dimension $k$ has exactly $k$ minimal tropical generators. Furthermore, they are $k$ columns of $c$, while each of the other $n-k$ columns are tropical multiples of one of these. Thus, it is natural to take the unique columns of $c$ as \emph{the} ordered set of tropical generators of $P$.  

\begin{defn}
Consider a polytrope $Pol(c)$ in $\mathbb{TP}^{n-1}$. Suppose that $c$ has $k$ unique columns $c_{i_1}, \ldots, c_{i_k}$, for $1 \leq i_1 < i_2 < \ldots < i_k \leq n$, $k \in [n]$. The tropical type of a polytrope is its type as $\tconv(c_{i_1}, \ldots, c_{i_k})$. 
\end{defn}

The goal of this paper is to classify polytropes up to their tropical types. By Proposition~\ref{prop:polytrope.generators}, these tropical types are tied to the shortest path ideal $I_s$. A consequence of Proposition \ref{prop:mikebernd.grobner} is the following. 

\begin{prop}\label{prop:polytrope.is}
Consider polytropes $Pol(c)$, $Pol(c')$ in $\mathbb{TP}^{n-1}$. Then they have the same tropical type if and only if $\ini_c(I_s) = \ini_{c'}(I_s)$. 
\end{prop}
\begin{proof}
By Proposition \ref{prop:polytrope.generators}, $Pol(c) = tconv(c) = \{y: y = c \odot z \mbox{ for some } z \in \R^n\}$. Since $c \in \mathcal{P}_n$, $c = c \odot c$, and $c_{ii} = 0$ for all $i \in [n]$. So in particular, we can take $z = y$, so $\ini_c(I_s) = \ini_c(I_t)$. Thus, by Proposition \ref{prop:mikebernd.grobner}, $Pol(c)$ and $Pol(c')$ have the same tropical type if and only if $\ini_c(I_s) = \ini_{c'}(I_s)$.
\end{proof}

As mentioned above, a polytrope of dimension $k$ has exactly $k$ minimal generators. So for $k < n$, a polytrope of dimension $k$ in $\mathbb{TP}^{n-1}$ is just a full-dimensional polytrope of $\mathbb{TP}^k$ embedded into $\mathbb{TP}^{n-1}$. Thus, we shall restrict our study to full-dimensional polytropes. A classical result \cite{BCOQ} states that for $c \in \mathcal{P}_n$, the $i$-th column $c_i$ is a tropical scalar multiple of the $j$-th column $c_j$ if and only if there exists a cycle of value zero going through $i$ and $j$. In particular, columns of $c_i$'s are distinct if and only if there are no zero cycles involving two nodes or more. In other words,

\begin{lem}\label{lem:full.dim}
A polytrope $Pol(c)$ is full-dimensional if and only if $c \in \mathcal{P}_n \cap R_n^\circ$. 
\end{lem}

Call the restriction of $\mathcal{GF}_n$ to the polytrope region $\mathcal{P}_n$ the \emph{polytrope complex} $\GFP$,
$$ \GFP = \bigcup_{c \in \R^{n^2-n}_{> 0} \cap \mathcal{P}_n} \overline{\mathcal{C}_c(I_s)}. $$
%Call the restriction of $\mathcal{GF}_n$ to $\mathcal{P}_n \cap R_n^\circ$ the \emph{interior polytrope complex} $\GFP^\circ$ 
%$$ \GFP^\circ = \bigcup_{c \in \R^{n^2-n}_{\geq 0} \cap \mathcal{P}_n \cap R_n^\circ} \overline{\mathcal{C}_c(I_s)}. $$
Note that by Lemma \ref{lem:grobner.region}, one has
$$ \GFP = \bigcup_{c \in \mathcal{P}_n} \overline{C_c(I_s)}.$$

\begin{thm}\label{thm:ds.analogue}
%Combinatorial types of polytropes in $\mathbb{TP}^{n-1}$ are in bijection with open cones of the Gr\"obner fan of $I_s$ intersected with the polytrope cone $\mathcal{P}_n$. 
Cones of $\GFP$ are in bijection with tropical types of polytropes in $\mathbb{TP}^{n-1}$. Furthermore, those cones of $\GFP$ in $R_n^\circ$ are in bijection with tropical types of full-dimensional polytropes in $\mathbb{TP}^{n-1}$. 
\end{thm}
\begin{proof}
By Proposition \ref{prop:polytrope.generators}, polytropes are tropical polytopes whose matrix of generators $c \in \mathcal{P}_n$. By Proposition \ref{prop:polytrope.is}, the types of such tropical polytopes are in bijection with cones of $\GFP$. This proves the first statement. Lemma \ref{lem:full.dim} proves the second. 
\end{proof}

From Theorem \ref{thm:ds.analogue}, enumerating tropical types of polytropes equals enumerating cones of $\GFP$. This is a much smaller polyhedral complex compared to $\mathcal{GF}_n$. %Somewhat surprisingly, if one can enumerate cones of $\GFP$, one can enumerate those of $\mathcal{GF}_n$ with little extra work (cf. Section \ref{sec:main} and \ref{sec:boundary}). Thus, $\GFP$ captures the essence of $\mathcal{GF}_n$. 
We conclude this section with an interpretation for the open cones of $\GFP$. As an ordinary polytope, a full-dimensional polytrope in $\mathbb{TP}^{n-1}$ has betwen $n$ and $\binom{2n-2}{n-1}$ vertices. A polytrope $Pol(c)$ in $\TP^{n-1}$ is \emph{maximal} if it has $\binom{2n-2}{n-1}$ vertices as an ordinary polytope. 

\begin{lem}\label{lem:maximal}
The polytrope $Pol(c)$ is maximal if and only if $\mathcal{C}_c$ is an open cone of $\GFP$. In other words, open cones of $\GFP$ are in bijection with maximal polytropes in $\mathbb{TP}^{n-1}$. 
\end{lem}
\begin{proof}
From \cite[Corollary 25]{mikebernd}, $tconv(c)$ has the maximal number of vertices of $\binom{2n-2}{n-1}$ if and only if the Gr\"obner cone of $c$ defined with respect to the ideal $I_t$ is open. But for $c \in \mathcal{P}_n$, $tconv(c) = Pol(c)$, $\ini_c(I_s) = \ini_c(I_t)$. This means the Gr\"obner cone of $c$ defined with respect to $I_t$ coincides with that defined with respect to $I_s$. This proves the lemma. 
\end{proof}

\section{The Polytrope Complex}\label{sec:main}

With Theorem \ref{thm:ds.analogue}, one can use a Gr\"obner fan computation software such as \textsf{gfan} \cite{gfan} to enumerate polytropes. However, this does not necessarily elucidate the combinatorial structure of tropical types of polytropes. In this section we state and prove our main results on the structure of the polytrope complex $\GFP$, Theorem \ref{thm:S} and \ref{thm:main}. 
These state that $\GFP$ equals the refinement of the polyhedral complex $\mathcal{P}_n$ by the \emph{bipartite binomial} fan $\mathcal{BB}_n$. In particular, the open cones of $\GFP$, which are in bijection to maximal polytropes by Lemma \ref{lem:maximal}, are indexed by inequalities amongst bipartite binomials. As an example, we use this fact to compute the six types of maximal polytropes for $n = 4$ by hand.  

\subsection{The Polytrope Gr\"obner Basis}

\begin{defn}
The \emph{polytrope Gr\"obner basis} $PGB$ is the union of minimal reduced Gr\"obner bases over the cones of $\GFP$. 
\end{defn}

The polytrope Gr\"obner basis plays the role of the universal Gr\"{o}bner basis for the polytrope region, in the sense that it is a Gr\"obner basis with respect to any term ordering~$\succ_c$ for $c \in \mathcal{P}_n$. The minimal condition means that elements of PGB are not redundant. That is, for each $f \in PGB$, there exists a cone $\mathcal{C}_c$ in $\GFP$ such that $\ini_c(f)$ is not a redundant generator of $\ini_c(I_s)$. The reduced condition implies that terms in the polytrope Gr\"obner basis are of the form $x^{u_+} - x^{u_-}$, where $u$ is a circuit of $A$. We claim that these terms fall into either one of the following categories: triangle and bipartite.

\begin{defn}[Bipartite monomials and binomials] \label{defn:m.bipartite}
For an integer $m \geq 2$, let $\mathbb{S}_m$ be the set of permutations on $m$ letters, $\Sigma_m \subset \mathbb{S}_m$ be the set of cyclic permutations. Let $K = (k_1 \leq k_2 \leq \ldots \leq k_m)$, $L = (l_1 \leq l_2 \leq \ldots \leq l_m) \subset [n]$ be two sequences of $m$ indices, not necessarily distinct, such that $K \cap L = \emptyset$.
For $\sigma \in \mathbb{S}_m$, $\tau \in \Sigma_m$, define
\begin{equation}\label{eqn:m.bipartite.binomial} 
    u_+ := k_1\to \sigma(l_1), \ldots, k_m \to \sigma(l_m), \hspace{1em} u_- := k_1 \to (\tau\circ\sigma)(l_1), \ldots, k_m \to (\tau\circ\sigma)(l_m).
\end{equation} 
If $(K,\sigma,\tau,L)$ is such that $(u_+,u_-)$ defined above is a circuit of $A_s$, say that $(u_+,u_-)$ is a bipartite binomial, and $u_+$, $u_-$ are bipartite monomials. 
\end{defn}
\begin{ex}\label{ex:n4.maximal} For $n = 4$, there are twelve bipartite monomials and six bipartite binomials. Figure \ref{fig:six} shows the six bipartite binomials, identified with the graphs of $u_+$ and $u_-$.
\vskip-0.5em
\begin{figure}[h]
\includegraphics[width=\textwidth]{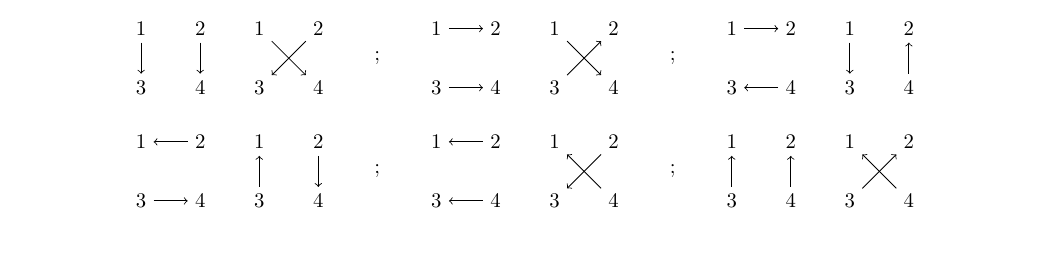}
\vskip-1em
\caption{The six bipartite binomials for $n = 4$.} \label{fig:six}
\end{figure}
\end{ex}

%\begin{ex}
%Consider $n = 10$, $m = 6$. With repeated indices. %Then $\Sigma_3 = \{(1)(2)(3), (1\,\, 2\,\, 3 ), (1\,\, 3\,\, 2)\}$. There are $\binom{6}{3}$ source-sink pairs $(K,L)$. For a fixed pair, say $K = (1, 4, 5)$, $L = (2, 3, 6)$, there are three $3$-bipartite monomials below. 
%(TODO).
%\end{ex}

\begin{cor}
There are finitely many bipartite binomials. 
\end{cor}
\begin{proof}
The bipartite binomials is a subset of the set of circuits of $A_s$, which is a matrix of dimension $n \times (n^2-n)$. So there are at most $\binom{n^2-n}{n}$ many circuits. 
\end{proof}

\begin{prop}\label{prop:ugb.polytrope}
The polytrope Gr\"obner basis is the set of binomials of the form $x^{u_+} - x^{u_-}$, where the pair $(u_+,u_-)$, identified with their graphs, ranges over the following sets:
\begin{itemize}
    \item \emph{Triangles:} $u_+ = i \to k \to j$, $u_- = i \to j$ for all distinct $i,k,j \in [n]$.
    \item \emph{Bipartite:} $(u_+,u_-)$ is a circuit of the form (\ref{eqn:m.bipartite.binomial}) for some $(K,\sigma,\tau,L)$ in Definition~\ref{defn:m.bipartite}. 
\end{itemize}
\end{prop}
\begin{proof}
Let $(u_+,u_-)$ be a circuit of $A_s$. Then $x^{u_+} - x^{u_-}$ is in the polytrope Gr\"obner basis if and only if for some $c \in \mathcal{P}_n$, either $u_+$ or $u_-$ is the optimal transport plan with cost $c$ subject to the net outflow constraint at each node (the Gr\"obner condtion), and that the optimality of these plans is not implied by other terms in the polytrope Gr\"obner basis (the minimality condition). 
First we show that our candidate set of PGB indeed consists of polynomials in the Gr\"obner basis, and that they are not redundant. For each pair $i,j \in [n]$, $i \to j$ is the shortest path from $i$ to $j$ on $\mathcal{P}_n$. Furthermore, for each $k \in [n], k \neq i,j$, there is a face of $\mathcal{P}_n$ defined by $c_{ij} = c_{ik} + c_{kj}$. Thus, the triangle terms are in the PGB. 
Now consider a bipartite binomial $x^{u_+} - x^{u_-}$. Define $c \in \R^{n^2 - n}$ via 
$$ c_{ij} = \left\{\begin{array}{ccc} 1 &\mbox{ if } & i \to j \notin u_+, i \to j \notin u_- \\
0 & \mbox{ else } & 
\end{array} \right.  $$
for all $i,j \in [n], i \neq j$. Then $c \in \mathcal{P}_n$, and $x^{u_+} - x^{u_-}$ is a non-redundant generator of $\ini_c(I_s)$. So the bipartite binomials are also contained in the PGB. 

Now we claim that given the triangles and bipartite binomials, any other circuit must be redundant. Let $(u_+,u_-)$ be a circuit of $A_s$. Since $(u_+,u_-)$ is in the kernel of $A_s$, each node in the graph of $u_+$ and $u_-$ must have the same net outflow. This partitions the support of $u_+$ and $u_-$ into three sets: the sources (those with positive net outflow), the sinks (those with negative net outflow), and the transits (those with zero net outflow). We now consider all possible outflow constraints. 

\begin{itemize}
  \item \textbf{One sink, one source}. Suppose there is exactly one source $i$ and one sink $j$.
This means $u_+, u_-$ are paths from $i$ to $j$. Consider further subcases based on the length of the paths $u_+, u_-$.
\begin{itemize}
  \item $u_-$ is $i \to j$, and $u_+$ is $i \to k \to j$. Then $(u_+,u_-)$ is a triangle term. 
  \item $u_-$ is $i \to k \to j$, and $u_+$ is $i \to k' \to j$. Since $i \to j$ must be a shortest path on $\mathcal{P}_n$, this means $(u_+,u_-)$ is made redundant by the triangles $(u_+,i\to j)$ and $(u_-,i\to j)$.
  \item Either $u_+$ or $u_-$ is of the form $i = i_0 \to i_1 \to \ldots \to i_{m-1} \to i_m = j$ for $m \geq 3$. Then it is a shortest path if and only if $i_r \to i_{r+1} \to i_{r+2}$ is a shortest path from $i_r$ to $i_{r+2}$ for all $r = 0, \ldots, m-2$. Thus, $(u_+,u_-)$ is made redundant by the triangles $(i_r \to i_{r+1} \to i_{r+2}, i_r \to i_{r+2})$ for $r = 0, \ldots, m-2$.
\end{itemize}
  \item \textbf{One source or one sink}.
Suppose there are $s \geq 2$ sinks, $1$ source. Since the constraints are integral, one can decompose any transport plan as the union of $s$ plans, one for each sink-source pair. So this reduces to the one sink one source case. The same reduction applies when there are $s \geq 2$ sources, $1$ sink. 
  \item \textbf{More than one sources and sinks}.
Suppose there are more than one sources and sinks. Let $(u_+,u_-)$ be a circuit of $A_s$ satisfying the constraint on the number of sources and sinks. Consider the following cases.
\begin{itemize}
  \item Either $u_+$ or $u_-$ contain a path $i \to j \to \ldots \to k$ of length at least two. One can replace it with the path $i \to k$ to form $u'$. Then the new binomial $(u_+,u')$ (or $(u',u_-)$) is a circuit of $A_s$, and it makes $(u_+,u_-)$ redundant. 
  \item All paths in $u_+$ and $u_-$ are of length 1, that is, each $u_+$ and $u_-$ is a bipartite graph. Since $(u_+,u_-)$ is in the kernel of $A_s$, the graphs of $u_+$ and $u_-$ must have the same number of edges, say, $m$ edges, for $m \geq 2$. Thus, we can write $u_+ = (K,\sigma,L)$, and $u_- = (K,\sigma',L)$ for $\sigma,\sigma' \in \mathbb{S}_m$, $K \cap L = \emptyset$, where $K$ and $L$ may have repeated indices. Write $\sigma' = \tau \circ \sigma$ for some $\tau \in \mathbb{S}_m$. Now we consider further subcases.
  \begin{itemize}
    \item $\tau$ has one cycle, that is, it is a cyclic permutation. Then $(u_+,u_-)$ is a bipartite binomial. 
    \item $\tau$ has more than one cycle. Then the induced bipartite pair $(u_+',u_-')$ on each cycle is another bipartite binomial with strictly smaller support. This contradicts the fact that $(u_+,u_-)$ is a circuit. 
  \end{itemize}
\end{itemize}
\end{itemize}
\end{proof}

\begin{defn}
For a set $S$ of triangle and bipartite monomials, define the cone $\mathcal{C}_S \subset \mathcal{P}_n$ as follows. For $c \in \mathcal{C}_S$, for each bipartite monomial $(K,\sigma,L) \in S$ with $|K|=|L|=m$, 
\begin{equation}\label{eqn:cone.from.plan1}
c_{k_1\sigma(l_1)} + \ldots + _{k_m\sigma(l_m)} < c_{k_1\tau(l_1)} + \ldots + c_{k_m\tau(l_m)} \mbox{ for all } \tau \in \mathbb{S}_m, \tau \neq \sigma,
\end{equation}
for each triangle monomial $i \to j \to k \in S$,
\begin{equation}\label{eqn:cone.from.plan2}
c_{ij} + c_{jk} = c_{ik},
\end{equation}
and for all distinct triples $i,j,k \in [n]$, $c_{ij} + c_{jk} \geq c_{ik}$. Say that $S$ is compatible if $\mathcal{C}_S \neq \emptyset$. 
\end{defn}

\begin{thm}\label{thm:S}
The map $S \mapsto \mathcal{C}_S$ is a bijection between compatible sets of triangle and bipartite monomials and cones of $\GFP$. 
\end{thm}
\begin{proof}
A monomial $x^u$ is not in $\ini_c(I_s)$ if and only if $u$ is an optimal transport plan amongst those with the same sources and sinks. By Proposition~\ref{prop:ugb.polytrope}, each relatively open cone $\mathcal{C}_c$ of $\GFP$ is defined by a unique set $S$ of triangle and bipartite monomials which are optimal transport plans amongst those with the same sinks and sources.
The optimal of terms in $S$ is expressed in (\ref{eqn:cone.from.plan1}) and (\ref{eqn:cone.from.plan2}). Therefore, $\mathcal{C}_c = \mathcal{C}_S$. Conversely, if $S$ is compatible, then any $c \in \mathcal{C}_S$ induces the same ordering on the binomials in the PGB, and so $\mathcal{C}_S$ is a non-empty cone of $\GFP$. This establishes the bijection claimed.
\end{proof}

\subsection{The fan structure of $\GFP$}
In this section, we translate the results of the previous section into a statement about the geometry of the polyhedral complex $\GFP$. Fix sources $K$, sinks $L$, with $|K| = |L| = m \geq 2$. Associate with each pair $\sigma,\tau \in \mathbb{S}_m$, $\sigma \neq \tau$ a hyperplane in $\R^{n^2-n}$ whose normal vector is $u_+ - u_-$ for $u_+, u_-$ defined as
$$  u_+ := k_1\to \sigma(l_1), \ldots, k_m \to \sigma(l_m), \hspace{1em} u_- := k_1 \to \tau(l_1), \ldots, k_m \to \tau(l_m). $$
Let $\mathcal{AB}_n(K,L)$ denote the arrangement of all hyperplanes ranging over all such pairs~$\sigma,\tau$. Note that each chamber of $\mathcal{AB}_n(K,L)$ defines a linear ordering on the $m!$ elements of $\mathbb{S}_m$. Say that two such linear orders are equivalent if they have the same minimum. This induces an equivalence relation $\sim_{\min}$ on the chambers of $\mathcal{AB}_n(K,L)$. Let $\mathcal{BB}_n(K,L)$ be the polyhedral complex obtained by removing faces between adjacent cones which are equivalent under $\sim_{\min}$. Then $\mathcal{BB}_n$ has at most $m!$ full-dimensional cones, indexed by the permutation $\sigma \in \mathbb{S}_m$ that achieves the minimum order amongst the $m!$ elements of $\mathbb{S}_m$. That is, the cone corresponds to $\sigma \in \mathbb{S}_m$ is defined by~(\ref{eqn:cone.from.plan1}). By construction, one can check that $\mathcal{BB}_n(K,L)$ is a fan coarsening of $\mathcal{AB}_n(K,L)$. 
%This fan has been studied in ranking problem .... Volkmar?

\begin{defn}
The \emph{bipartite binomial fan $\mathcal{BB}_n$} is the refinement of the fans $\mathcal{BB}_n(K,L)$, and the \emph{bipartite binomial arrangement $\mathcal{AB}_n$} is the refinement of the arrangements $\mathcal{AB}_n(K,L)$, over all pairs of sources and sinks $(K,L)$ such that there exists some bipartite monomial with these sources and sinks. 
\end{defn}

The name `bipartite binomial arrangement' stems on the fact that $\mathcal{AB}_n$ is an arrangement of bipartite binomials which appear in the polytrope universal basis. Since bipartite binomials are a subset of the set of circuits of $A_s$, $\mathcal{AB}_n$ is a coarsening of the circuit arrangement of $A_s$ studied in \cite{st94}. %All hyperplanes of $\mathcal{AB}_n$ share the same subspace 
%$$ \lin(R_n) + \mathsf{span}(1, \ldots, 1).  $$
%Modulo this subspace, $\mathcal{AB}_n$ is a central hyperplane arrangement. For fixed $n$, it may be possible to count chambers of $\mathcal{AB}_n$. However, from construction, it is clear that this would be a large overestimate of the number of chambers of $BB_n$.

\begin{ex} For $n = 4$ and $n=5$, $\mathcal{BB}_n = \mathcal{AB}_n$, and this is the arrangement of hyperplanes
$$\{ c \in \R^{n^2-n}: c_{ik} + c_{jl} - c_{il} - c_{jk} = 0\} $$
for each tuple of distinct indices $i,j,k,l \subset [n]$. 
\end{ex}

\begin{ex}
Suppose $K = (1,2,3)$, $L = (4,5,6)$. There are $3! = 6$ bipartite monomials with sources $K$ and sinks $L$, shown in Figure \ref{fig:ab} below.
\begin{figure}[h]
\includegraphics[width=\textwidth]{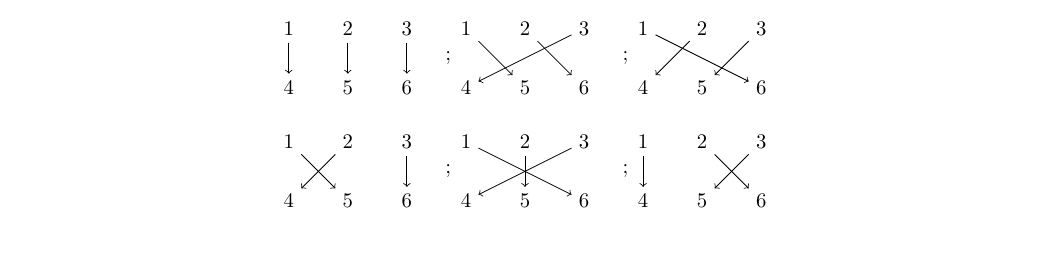}
\vskip-1em
\caption{The six bipartite monomials with sources $(1,2,3)$ and sinks $(4,5,6)$.} \label{fig:ab}
\end{figure}

Each pair of monomials generate a hyperplane. For example, the pair of top left monomials defines the hyperplane 
$$c_{14} + c_{25} + c_{36} - (c_{15} + c_{26} + c_{34}) = 0. $$
The arrangement $\mathcal{AB}_n(K,L)$ is generated by the $\binom{3!}{2} = 15$ hyperplanes from these pairs. In comparison, $\mathcal{BB}_n(K,L)$ has $3! = 6$ full-dimensional cones, each given by 5 inequalities. For instance, the cone indexed by the first monomial in Figure \ref{fig:ab} is given by
\begin{align*}
c_{14} + c_{25} + c_{36} 
&< c_{15} + c_{26} + c_{34}, c_{16} + c_{24} + c_{35}, c_{14} + c_{26} + c_{35}, c_{16} + c_{25} + c_{34}, c_{15} + c_{24} + c_{36}. 
\end{align*}
\end{ex}

\begin{thm}\label{thm:main}
The Gr\"obner fan on the polytrope region $\GFP$ equals the refinement of the polyhedral complex $\mathcal{P}_n$ by $\mathcal{BB}_n$.
\end{thm}
\begin{proof}
By the discussion succeeding Proposition~\ref{prop:ugb.polytrope}, cones of $\GFP$ are in bijection with cones $\mathcal{C}_S$ indexed by compatible sets $S$ of triangle and bipartite monomials. By construction, the cones of $\mathcal{P}_n$ are in bijection with all compatible sets of triangle monomials, and the cones of $\mathcal{BB}_n$ \emph{over $\R^{n^2 - n}$} are in bijection with all compatible sets of bipartite monomials. Thus, the conclusion would follow if we can show that every cone of $\mathcal{BB}_n$ has non-empty intersection with $\mathcal{P}_n^\circ$. Since $\mathcal{BB}_n$ is the fan coarsening of $\mathcal{AB}_n$, it is sufficient to show that every cone of $\mathcal{AB}_n$ has non-empty intersection with $\mathcal{P}_n^\circ$. The lineality space of $\mathcal{AB}_n$ is
$$ \lin(R_n) + \mathsf{span}(1, \ldots, 1),  $$
where $\lin(R_n)$ is defined in (\ref{eqn:vn}). Over $\R^{n^2-n} \backslash \lin(R_n)$, $\mathcal{P}_n$ is a pointed cone containing the ray $(1, \ldots, 1)$ in its interior. Let us further modulo the span of this ray. Then $\mathcal{AB}_n \backslash (\lin(R_n) + \mathsf{span}(1, \ldots, 1))$ is a central hyperplane arrangement, and $\mathcal{P}_n^\circ \backslash (\lin(R_n) + \mathsf{span}(1, \ldots, 1))$ is an open neighborhood around the origin. Thus every cone of $\mathcal{AB}_n$ has non-empty intersection with $\mathcal{P}_n^\circ$. This proves the claim. 
\end{proof}

\begin{cor}\label{cor:maximal.equals.open.chambers}
The number of combinatorial tropical types of maximal polytropes in $\mathbb{TP}^{n-1}$ is precisely the number of equivalence classes of open cones $\mathcal{BB}_n$ up to action by $\mathbb{S}_n$. 
\end{cor}
%The above is a consequence of Lemma \ref{lem:maximal} and Proposition \ref{thm:main}. The case $n = 4$ can practically be done by hand. A formula for general $n$ is an open problem.

\begin{ex}[Maximal polytropes for $n = 4$]\label{ex:n4.maximal}
Number the binomials in Figure \ref{fig:six} from left to right, top to bottom. Here $\mathcal{BB}_4$ equals the hyperplane arrangement $\mathcal{AB}_4$. An open chamber of $\mathcal{BB}_4$ is a binary vector $z = \{\pm 1\}^6$, with $z_i = +1$ if in the $i$-th binomial, the left monomial is smaller than the right monomial. For example, $z_2 = +1$ correspond to the inequality $c_{12} + c_{34} < c_{14} + c_{32}$. There are at most $2^6 = 64$ open chambers in $\mathcal{BB}_4$. Not all of 64 possible values of $z$ define a non-empty cone. Indeed, the six normal vectors satisfy exactly one relation:
\vskip-2em
\begin{figure}[h]
\includegraphics[width=\textwidth]{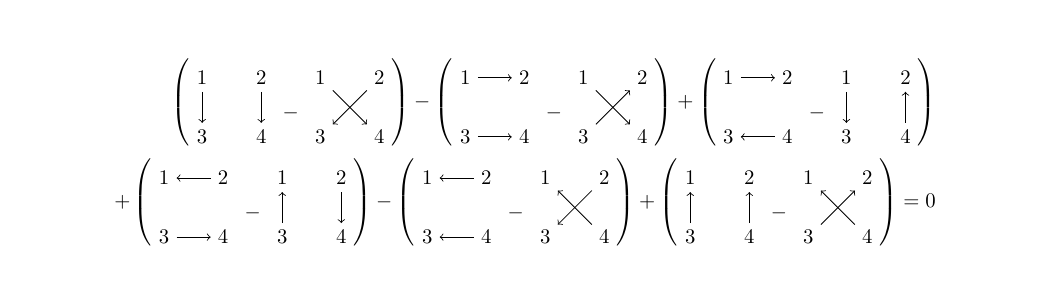}
\vskip-2em
\caption{The relation amongst the two-bipartite binomials for $n = 4$.}
\label{eqn:relation}
\end{figure}

This means $(1,-1,1,1,-1,1)$ and $(-1,1,-1,-1,1,-1)$ define empty cones. Thus there are 62 open chambers of $\mathcal{BB}_4$, correspond to 62 tropical types of maximal polytropes. The symmetric group $\mathbb{S}_4$ acts on the vertices of a polytrope $Pol(c)$ by permuting the labels of the rows and columns $c$. This translates to an action on the chambers of $\mathcal{BB}_4$. Up to the action of $\mathbb{S}_4$, we found six symmetry classes of chambers, corresponds to six combinatorial tropical types of maximal polytropes. Table \ref{tab:max4} shows a representative for each symmetry class and their orbit sizes. The first five corresponds to the five types discovered by Joswig and Kulas, presented in the same order in \cite[Figure 5]{JoswigK10}. The class of size 12 was discovered by Jimenez and de la Puente \cite{jimenez2012six}.

\begin{table}[h]
\begin{tabular}{|c|c|}
\hline
Representative & Orbit size \\
\hline
$(1,1,1,-1,1,1)$ & 6 \\
$(-1,1,1,-1,-1,1)$ & 8 \\
$(1,1,1,1,1,-1)$ & 6 \\
$(-1,1,-1,-1,-1,1)$ & 24 \\
$(-1,-1,1,1,-1,-1)$ & 6 \\
$(-1,-1,-1,-1,1,-1)$ & 12 \\
\hline
\end{tabular}
\vskip0.5em
\caption{Representatives and orbit sizes of the six maximal polytropes in $\mathbb{TP}^3$.}
\label{tab:max4}
\end{table}
\end{ex}

\section{Polytropes enumeration: algorithms, results and summary}\label{sec:algorithm}

\subsection{Algorithms and results}
We have two algorithms for enumerating combinatorial tropical types of full-dimensional polytropes in $\mathbb{TP}^{n-1}$. Recall that we are enumerating cones of $\GFP$ which are \emph{not} on $\partial R_n$, up to symmetry induced by $\mathbb{S}_n$. The two algorithms differ only in the first step of computing $\GFP$. The first computes $\GFP$ using a Gr\"{o}bner fan computation software such as \textsf{gfan} \cite{gfan}. In the second algorithm, one computes the polyhedral complex $\mathcal{P}_n$ first, then computes the refinement of its cones by $\mathcal{BB}_n$. Given $\GFP$, one can then remove all cones in $\partial R_n$. We find such cones as follows: for each cone, pick a point $c$ in the interior and compute the minimum cycle in the undirected graph with edge weights $c_{ij}$. If the minimum cycle is zero, this point comes from a cone on $\partial R_n$, and thus should be removed. A documented implementation of the first algorithm, with examples for $n = 4$ and input files for $n = 4, 5$ and $6$, is available at \url{https://github.com/princengoc/polytropes}. 

For $n = 4$, we found 1026 symmetry classes of cones in $\GFP$, of which 13 are in $\partial R_n$. Thus, there are 1013 combinatorial tropical types of polytropes in $\mathbb{TP}^3$. Table~\ref{tab:jk.table} classifies the types by the number of vertices of the polytrope. This corresponds to the first column of \cite[Table 1]{JoswigK10}.

\begin{table}[h]
\begin{tabular}{|c|c|c|c|c|c|c|c|c|c|c|c|c|c|c|c|c|c|c}
\hline
\# vertices & 4 & 5 & 6 & 7 & 8 & 9& 10& 11& 12& 13& 14& 15& 16& 17& 18& 19& 20 \\
\hline
\# types & 1 & 1 & 5 & 6 & 34 & 38 & 81 & 101 & 151 & 144 & 154 & 116 & 92 & 46 & 28 & 9 & 6 \\
\hline
\end{tabular}
\vskip1em
\caption{Combinatorial tropical types of full-dimensional polytropes in $\mathbb{TP}^3$, grouped by total number of vertices.}
\label{tab:jk.table}
\end{table}

We also implemented the second algorithm for $n = 4$. We found 273 equivalence classes of cones of the polyhedral complex $\mathcal{P}_4$. Table \ref{tab:gz} groups them by the number of equivalence classes of cones in the refinement $\mathcal{P}_4 \wedge \mathcal{BB}_4$ that they contain. Altogether, we obtain 1013 equivalence classes, agreeing with the first output.

\begin{table}[h]
\begin{tabular}{l|c|c|c|c|c|c|c|c|c|c|c|c}
$\#$ $F$ & 123 & 10 & 89 &19& 2 & 19 & 2& 3 & 3 & 1& 1 & 1 \\
\hline
$\#$ $(F,z)$ & 1 & 2 & 3 & 5 & 6 & 9 & 15 & 18 & 27 & 37 & 42 & 81
\end{tabular}
\vskip0.5em
\caption{Equivalence classes of cones $F$ of $\mathcal{P}_4$, grouped by the number of equivalence classes of cones in $\left.\mathcal{GF}_4\right|_\mathcal{P}$ that they correspond to. For instance, up to symmetry, there are 123 cones of $\mathcal{P}_4$ which are not subdivided by $\mathcal{BB}_4$, and thus they each yield one cone of $\left.\mathcal{GF}_4\right|_\mathcal{P}$. Up to symmetry, there are 10 cones of $\mathcal{P}_4$ which are subdivided into two by $\mathcal{BB}_4$, 89 cones subdivided into 3, and so on. The sum $123 \cdot 1 + 10 \cdot 2 + 89 \cdot 3 + \ldots + 1 \cdot 81$ equals 1013, agreeing with the number of equivalence classes of polytropes computed by~\textsf{gfan} \cite{gfan}.} \label{tab:gz}
\end{table}

The polytrope complex $\GFP$ grows large quickly. For $n = 5$, there are $27248$ open cones, correspond to combinatorial tropical types of maximal polytropes in $\mathbb{TP}^5$. This is clearly much bigger than six, the corresponding number for $n = 4$. The fan $\mathcal{BB}_5$ is the arrangement $\mathcal{AB}_5$ of $5 \binom{4}{2} = 30$ bipartite binomial hyperplanes. 
%For $n=4$ and $n=5$, $\mathcal{BB}_n$ is the hyperplane arrangement $\mathcal{AB}_n$, 
The orderings of the bipartite binomials which lead to empty cones of $\mathcal{AB}_n$ are precisely those which contain a circuit of the oriented matroid associated with $\mathcal{AB}_n$ \cite{orientedMatroid}. Up to permutation, there are 11 circuits. We list them on \url{https://github.com/princengoc/polytropes/output/n5relations.txt} in a format analogous to that in Figure \ref{eqn:relation}. 

%COMMENT: (TODO) should be able to run gfan multiple times to get cones of $\GFP$ for $n = 5$. First compute all the open cones (done). Now compute all those restricted to each smaller cone of $P_5$. 

Using \textsf{gfan} \cite{gfan}, we could not compute all cones of $\left.\mathcal{GF}_5\right|_\mathcal{P}$ or the open cones of $\left.\mathcal{GF}_6\right|_\mathcal{P}$ on a conventional desktop. However, we believe that such computations should be possible on more powerful machines. The open cones of $n = 6$ is particularly interesting, since this is the smallest $n$ for which $\mathcal{BB}_n$ is a strict coarsening of $\mathcal{AB}_n$. 

\subsection{Summary and open problems}
Tropical types of polytropes in $\mathbb{TP}^{n-1}$ are in bijection with cones of the polyhedral complex $\GFP$. This complex is the restriction of a certain 
Gr\"{o}bner fan $\mathcal{GF}_n \subset\R^{n^2-n}$ to a certain cone $\mathcal{P}_n$. We showed that $\GFP$ equals the refinement of several fans. These fans are significantly smaller than $\mathcal{GF}_n$, giving a computational advantage over brute force approaches. We utilized these results to enumerate all combinatorial tropical types of full-dimensional polytropes in $\mathbb{TP}^3$, and those of maximal polytropes in $\mathbb{TP}^4$. 

Theorem \ref{thm:S} establishes a bijection between cones of $\GFP$ and compatible sets of triangles and bipartite monomials. The central open question is to give an intrinsic characterization of this compatibility. This question has been answered for triangle monomials alone in \cite{tran.combi}, where sets of compatible triangles are indexed by a certain collection of trees. However, we do not know of such characterizations for the bipartite monomials. A characterization for compatibility amongst the bipartite monomials would potentially allow one to enumerate the open cones of $\mathcal{BB}_n$ up to $\mathbb{S}_n$ action. This number is precisely the number of tropical types of maximal polytropes. There are obvious requirements, such as if $(K,\sigma,L)$ is in the set, then any bipartite subgraph of $(K,\sigma,L)$ must also be in the set. However, this requirement alone is not enough. For instance, for $n = 4$, of the $64$ sets of bipartite monomials that satisfy the subgraph requirement, only $62$ define non-empty cones and thus are compatible (cf. Example \ref{ex:n4.maximal}). Even this example is not representative, as in this case, $\mathcal{BB}_4$ is the arrangement $\mathbb{AB}_4$, while in general $\mathcal{BB}_n$ is not a hyperplane arrangement. 

\subsection{Acknowledgements} I sincerely thank Bernd Sturmfels, Josephine Yu and Anders Jensen for stimulating discussions. Special thanks to Michael Joswig and Katja Kulas for their inspiring pictures of polytropes. I would like to thank two anonymous referees for their detailed reading and constructive comments of an earlier draft. This work was supported by an award from the Simons Foundation ($\#197982$ to The University of Texas at Austin).

\bibliographystyle{plain}
\bibliography{references}

\end{document}

%% file: preamble.tex
\usepackage[body={6.3in,9.3in}]{geometry}
\usepackage{verbatim}
\usepackage[ansinew]{inputenc}

\usepackage{textcomp}
\usepackage{latexsym}
\usepackage[usenames]{color}
\usepackage{amsmath,amsfonts,amssymb,amsthm}
\usepackage{graphicx}
\usepackage{longtable}
\usepackage{bbm}
\usepackage{algorithm, algorithmicx, algpseudocode}
\usepackage{hyperref}

\usepackage{setspace}
\usepackage{array}

\def\qed{\hfill{\raggedleft{\hbox{$\Box$}}} \smallskip}

\def\R{\mathbb{R}}

\newcommand{\ds}{\displaystyle}

\theoremstyle{plain} \newtheorem{lem}{Lemma}
\theoremstyle{plain} \newtheorem{prop}[lem]{Proposition}
\theoremstyle{plain} \newtheorem{thm}[lem]{Theorem}
\theoremstyle{plain} \newtheorem{cor}[lem]{Corollary}
\theoremstyle{plain} 
\theoremstyle{plain} 
\theoremstyle{definition} \newtheorem{defn}[lem]{Definition}
\theoremstyle{definition}
\theoremstyle{definition} 
\theoremstyle{definition} 
\theoremstyle{definition}\newtheorem{ex}[lem]{Example}

\newlength\savedwidth